\documentclass[10pt, a4paper, reqno, english]{amsart}
\usepackage[utf8]{inputenc}
\usepackage{amsmath,amsthm,amssymb}

\def\cb{\mathbf{c}}
\def\xb{\mathbf{x}}
\def\yb{\mathbf{y}}
\def\zb{\mathbf{z}}
\def\Xb{\mathbf{X}}
\def\Yb{\mathbf{Y}}
\def\Zb{\mathbf{Z}}
\newcommand\ddd{\,\mathrm{d}}

\newcommand\NN{\mathbb{N}}

\theoremstyle{plain}
\newtheorem{theorem}{Theorem}
\newtheorem{prop}[theorem]{Proposition}

\newcommand\bigwhere[2]{\left\{#1:\ \begin{aligned}#2\end{aligned}\right\}}

\title{The exceptional set in the abc conjecture}
\author[C. Bernert]{Christian Bernert}
\address{Institut f\"ur Algebra, Zahlentheorie und Diskrete
  Mathematik, Leibniz Universit\"at Hannover,
  Welfengarten 1, 30167 Hannover, Germany}
\email{bernert@math.uni-hannover.de}

\date{\today}

\begin{document}

\begin{abstract}
    We study the size of the exceptional set in the $abc$ conjecture, improving on and simplifying  work of Browning, Lichtman and Teräväinen.
\end{abstract}

\maketitle

\section{Introduction}

The $abc$ conjecture predicts that for any given $\lambda \in (0,1)$, the number of solutions $(a,b,c)$ of $a+b=c$ in positive integers satisfying $\gcd(a,b,c)=1$ and $\text{rad}(abc)<c^{\lambda}$ is finite.
Let us denote by
\[N_{\lambda}(X)=\#\{(a,b,c) \in \NN^3: a+b=c, c \le X, \gcd(a,b,c)=1, \text{rad}(abc)<c^{\lambda}\}\]
the number of exceptions. Recently, Browning, Lichtman and Teräväinen ~\cite{BLT} were able to prove the bound $N_{\lambda}(X) \ll X^{0.66}$ for any fixed $\lambda \in (0,1.001)$. This improves on the \lq trivial bound\rq{} $N_{\lambda}(X) \ll X^{2\lambda/3+\varepsilon}$ (see \cite[Prop.~1.1]{BLT}).

Their proof proceeds by first bounding the exceptional set using four different methods, followed by a relatively long computational part that optimizes the ranges in which the various bounds are applied.
In this note, we show that two of those methods are already enough to deduce a better bound by a quick computation.

\begin{theorem} \label{thm}
    For $\lambda \in (0,2)$ we have
    \[N_{\lambda}(X) \ll X^{\frac{23\lambda+3}{40}+\varepsilon}.\]
    In particular, $N_1(X) \ll X^{0.65+\varepsilon}$ and $N_{\lambda}(X)=o(X^{2\lambda/3})$ for all $\lambda \in (9/11,3)$.
\end{theorem}
Note that for $\lambda \ge 2$, the optimal upper bound $N_{\lambda}(X) \ll X^{\lambda-1+\varepsilon}$ is known by work of Kane~\cite{Kane}.
We only use the most basic versions of the methods to showcase what might be the simplest way to get some nontrivial bound. At the end of this note, we sketch how one can obtain numerical improvements such as $N_1(X) \ll X^{8/13+\varepsilon}$ (where $8/13<0.6154$) by a refined version of the methods.

\subsection*{Notation} We write $x \sim X$ to denote the dyadic restriction $x \in [X,2X)$. We use the convention that an asymptotic bound like $\ll_{\varepsilon} X^{\varepsilon}$ is meant to hold for all $\varepsilon>0$.

\subsection*{Acknowledgements} We thank Tim Browning for comments on an earlier version of this paper, and Johannes Raitz von Frentz for carefully reading the manuscript.

\section{An anatomic reduction}

For $d \in \NN$, a vector $\cb \in \NN^3$ and parameters $\Xb=(X_1,\dots,X_d), \Yb=(Y_1,\dots,Y_d)$ and $\Zb=(Z_1,\dots,Z_d)$ in $\NN^d$, consider the quantity
\begin{equation*}
    B_d(\cb, \Xb,\Yb,\Zb)=\#\bigwhere{(\xb,\yb,\zb) \in \NN^{3d}}{& x_i \sim X_i, y_i \sim Y_i, z_i \sim Z_i,\\& c_1 \prod_i x_i^i+c_2\prod_i y_i^i=c_3 \prod_i z_i^i,\\& \gcd\left(c_1\prod_i x_i, c_2\prod_i y_i, c_3\prod_i z_i\right)=1}.
\end{equation*}
By the arguments from \cite[Section~2]{BLT} (and using the monotonicity of our bound), it suffices to prove that, for any fixed $d$, the bound
\[B_d(\cb, \Xb, \Yb, \Zb) \ll_{d,\varepsilon} X^{\frac{23\lambda+3}{40}+\varepsilon}\]
holds uniformly over all choices of positive integers $c_1,c_2,c_3, X_i,Y_i,Z_i$ satisfying $\prod_i (X_iY_iZ_i) \sim X^{\lambda}$ and 
\begin{equation*}
    \max\left(\prod_i X_i^i, \prod_i Y_i^i, \prod_i Z_i^i\right) \sim X.
\end{equation*}
Thus, from now on, we study bounds for $B_d(\cb, \Xb, \Yb, \Zb)$. For convenience, let us denote $P_i=X_iY_iZ_i$.

\section{The two bounds}

Our first bound arises from the geometry of numbers as a special case of \cite[Prop.~3.2]{BLT}. 

\begin{prop}\label{prop:geometry}
We have
\[B_d(\cb, \Xb, \Yb, \Zb) \ll_{d} \frac{X^{\lambda}}{P_1}+X^{\lambda-1}.\]
\end{prop}
\begin{proof}
    We fix all variables except $x_1,y_1,z_1$ to obtain a linear equation $a_1x_1+a_2y_1+a_3z_1=0$ with the $a_i$ depending on the fixed variables and on the $c_i$. By basic geometry of numbers, the number of primitive solutions is bounded by
    \[\ll 1+\frac{X_1Y_1Z_1}{\max(|a_1|X_1,|a_2|Y_1,|a_3|Z_1)} \ll 1+\frac{X_1Y_1Z_1}{X}.\] 
    Summing this over the choices of $x_i,y_i,z_i$ with $i \ge 2$, we obtain the result.
\end{proof}


Our second bound comes from Fourier analysis. It is a symmetrized version of \cite[Prop.~3.1]{BLT}.

\begin{prop}\label{prop:fourier}
    For $2 \le j \le d$, we have
    \[B_d(\cb, \Xb, \Yb, \Zb) \ll_{d,\varepsilon} X^{2\lambda/3+\varepsilon}/P_j^{1/6}.\]
\end{prop}
Note that this also recovers the \lq trivial bound\rq{}.
\begin{proof}
    By orthogonality of characters and Hölder's inequality, we have
    \[
    B_d(\cb, \Xb, \Yb, \Zb) \le \int_0^1 S_1(\alpha)S_2(\alpha)S_3(-\alpha) \ddd\alpha \le \left(\prod_{i=1}^3 \int_0^1 |S_i(\alpha)|^3 \ddd\alpha\right)^{1/3}
    \]
    with the exponential sums
    \[S_1(\alpha)=\sum_{x_i \sim X_i} e(\alpha c_1 x_1x_2^2\cdots x_d^d), \quad S_2(\alpha)=\sum_{y_i \sim Y_i} e(\alpha c_2y_1y_2^2\cdots y_d^d)\]
    and
    \[S_3(\alpha)=\sum_{z_i \sim Z_i} e(\alpha c_3z_1z_2^2\cdots z_d^d).\]
    We claim that 
    \begin{equation}\label{eq:second}
        \int_0^1 |S_1(\alpha)|^2 \ddd\alpha \ll_{d,\varepsilon} X^{\varepsilon} X_1X_2\cdots X_d
    \end{equation}
    and 
    \begin{equation}\label{eq:fourth}
        \int_0^1 |S_1(\alpha)|^4 \ddd\alpha \ll_{d,\varepsilon} X^{\varepsilon} \frac{(X_1X_2\cdots X_d)^3}{X_j}.
    \end{equation}
    Combined with an application of Cauchy-Schwarz and the analogous inequalities for $S_2$ and $S_3$, this will suffice to deduce the result.

    For the proof of the second moment bound \eqref{eq:second}, we note that, by orthogonality, the second moment counts the number of solutions to $X_1X_2^2\cdots X_d^d=X_1'X_2'^2\cdots X_d'^d$     with $x_i,x_i' \sim X_i$. But fixing all the $x_i$, the $x_i'$ are determined up to $O(X^{\varepsilon})$ many choices by the divisor bound. This proves \eqref{eq:second}.
    
    Turning to the fourth moment bound \eqref{eq:fourth}, we begin by applying Cauchy-Schwarz to the sum over $x_i, i \ne j$ in $S_1(\alpha)$ to obtain
    \[|S_1(\alpha)|^2 \le \prod_{i \ne j} X_i \sum_{x_i \sim X_i, x_j' \sim X_j} e\left(\alpha c_1 (x_j^j-x_j'^j)\prod_{i \ne j} x_i^i\right).\]
    Inserting this bound into our fourth moment and using orthogonality of characters, we find that the fourth moment is bounded by $\prod_{i \ne j} X_i$ times the number of solutions to
    \[
    (x_j^j-x_j'^j)\prod_{i \ne j} x_i^i=\prod_i (x_i'')^i-\prod_i (x_i''')^i
    \]
    with $x_i,x_i'', x_i''' \sim X_i$ and $x_j' \sim X_j$.
    It therefore suffices to show that the number of such solutions is bounded by $O(X^{\varepsilon} (\prod_i X_i)^2)$.
    This is certainly true for the cases where both sides are equal to zero, since we can choose the $x_i$ and $x_i''$ freely and then $x_j'=x_j$ is fixed, while the $x_i'''$ are determined up to a divisor function.
    
    In the remaining cases, we have at most $(\prod_i X_i)^2$ many choices for the $x_i''$ and $x_i'''$.  But then the left-hand side is fixed, hence the $x_i$ for $i \ne j$ are determined up to a divisor function. Finally, also $x_j^j-x_j'^j \ne 0$ is fixed, which fixes $x_j-x_j'$ up to another divisor function, and then also $x_j,x_j'$ up to $d$ many choices (being roots of a polynomial of degree $j-1$). This proves \eqref{eq:fourth} and therefore the proposition.
\end{proof}

\section{The endgame}


\begin{proof}[Proof of Theorem \ref{thm}]
    Suppose the contrary. Write $B_d(\cb,\Xb,\Yb,\Zb)= X^{2\lambda/3-\delta_0}$ and fix $\delta$ with $\delta_0<\delta<\frac{11\lambda-9}{120}$. By Proposition~\ref{prop:geometry}, we must have $P_1 \ll X^{\lambda/3+\delta}$. By Proposition~\ref{prop:fourier}, we must have $P_j \ll X^{6\delta}$ for all $j \ge 2$. However,
    \begin{equation*}
        X^{5\lambda-3} \ll \prod_i P_i^{5-i} \ll P_1^4P_2^3P_3^2P_4 \ll X^{4\lambda/3+40\delta}
    \end{equation*}
    which contradicts our upper bound for $\delta$.
\end{proof}

\emph{Remark.}     We briefly sketch how one can obtain a numerical improvement. As in \cite[Prop.~3.1]{BLT}, one can improve the Fourier bound in Proposition~\ref{prop:fourier} to have $\prod_{j \mid i} P_i$ instead of just $P_j$ in the denominator. It is thus possible to assume $P_2P_4P_6 \ll X^{6\delta}$ in the computation above.
    One can also improve on the geometry bound from Proposition~\ref{prop:geometry} by fixing everything except the linear and the quadratic variables $x_1,y_1,z_1,x_2,y_2,z_2$ to arrive at the bound $\frac{X^{\lambda+\varepsilon}}{P_1P_2}+X^{\lambda-1+\varepsilon} P_2$. Similarly, fixing everything except the linear and cubic variables, we obtain the bound $\frac{X^{\lambda+\varepsilon}}{P_1P_3}+X^{\lambda-1+\varepsilon} P_3^2$. Hence for $\lambda=1$, we have $P_1P_2, P_1P_3 \ll P^{1/3+\delta}$ in the computation above. 
   Inserting this into a bound for $\prod_i P_i^{7-i}$, one finds that $N_1(X) \ll X^{8/13+\varepsilon}$.

\bibliographystyle{alpha}

\bibliography{bibliography}

\end{document}